\documentclass[12pt]{article}
\usepackage[utf8]{inputenc}
\usepackage{amsmath}
\usepackage{caption}
\usepackage{epsfig}
\usepackage{graphicx}
\usepackage{latexsym}
\usepackage{amsfonts}
\usepackage{textcomp}
\usepackage{color,bm}
\usepackage{multirow}
\usepackage{verbdef}
\usepackage{mathrsfs}
\usepackage{mathtools}
\usepackage{comment}

\usepackage{bm}
\usepackage{subcaption}
\usepackage{epstopdf}
\usepackage{cite}
\usepackage{float}
\usepackage{tikz}
\usepackage{booktabs}
\usepackage{amssymb}
\usepackage{algorithm}

\usetikzlibrary{arrows,backgrounds,positioning,arrows.meta,chains,decorations.pathreplacing}

\usepackage{amsthm}
\usepackage{soul}

\newtheorem{proposition}{Proposition}[section]

\newtheorem{definition}{Definition}[section]

\def\R{{\mathbb R}}

\newcommand{\mK}{\mathsf{K}}
\newcommand{\mA}{\mathsf{A}}
\newcommand{\mP}{\mathsf{P}}
\newcommand{\mO}{\mathsf{O}}
\newcommand{\mI}{\mathsf{I}}

\title{Data-driven extrapolation via feature augmentation based on variably scaled thin plate splines}

\author{Rosanna Campagna$^{+}$, Emma  Perracchione$^{*}$\\
\small{$^+$Dept. of Mathematics and Physics, University of Campania ``L. Vanvitelli'', Italy} \\
 \small{$^{*}$Dept. of Mathematics DIMA, University of Genova, Italy}\\
  \small{ \texttt{rosanna.campagna@unicampania.it; perracchione@dima.unige.it} }
}

 \date{}

\begin{document}

\maketitle

\textbf{Keywords.}  Data-driven extrapolation, Feature augmentation, Radial basis functions, Variably scaled kernels, Thin plate splines, Ridge regression.\\
\textbf{65D05  \and 41A05  \and 65D10  \and 65D15}

\begin{abstract}

The data driven extrapolation requires the definition of a functional model depending on the available data and has the application scope of providing reliable predictions on the unknown dynamics. Since data might be scattered, we drive our attention towards kernel models that have the advantage of being meshfree. Precisely, the proposed numerical method makes use of the so-called  Variably Scaled Kernels (VSKs), which are introduced to implement a feature augmentation-like strategy based on discrete data.  Due to the possible uncertainty on the data and since we are interested in modelling the behaviour of the considered dynamics, we seek for a regularized solution by ridge regression. Focusing on polyharmonic splines, we investigate their implementation in the VSK setting and we provide error bounds in Beppo-Levi spaces. 
The performances of the method are then tested on functions which are common in the framework of the  Laplace transform inversion. Comparisons with Support Vector Regression (SVR) are also carried out and show that the proposed method is effective particularly since  it does not require to  train complex architecture constructions.

\end{abstract}

\section{Introduction}
The extrapolation of functional data is an attractive issue for both approximation and  information theory, due to the ever-growing need to predict states from noisy data. This is a challenging problem, indeed recovering information on the dynamics of data through extrapolation is a so-defined {\em hopelessly ill-conditioned problem} \cite{Bakas2019}.

While some authors refer to the extrapolation as super-resolution, i.e. the extrapolation of fine-scale details from low-resolution data \cite{candes2014towards}, we here consider the extrapolation \emph{out of samples}. For polynomial and rational analytic functions, an {\em extrapolant} can be  defined as a least-squares polynomial approximant \cite{DemanetT16}. 
Usually, the extrapolation from univariate data is defined by spline models \cite{SHETTY1991484}, which might give  unpredictable results. Indeed, the behaviour of the solution out of the reconstruction interval  is strongly forced by the  model gradient constraints at the boundaries. This implies that one usually recovers reliable approximations only locally, i.e. for short extensions outside the domain. %Applications in the context of inverse problems can be found in \cite{Campagna201986,RC2012}, where the authors investigate spline models for functional data deriving from Laplace transforms with {\em rational or exponential} decays.  Further in \cite{CampagnaAML}, focusing on Radial Basis Functions (RBFs), the same authors show that, for the  extrapolation of data sampled from rational monotonically decreasing functions, accurate approximations can be achieved via Polyharmonic Splines (PHSs) augmented with a polynomial term  that improves the accuracy of the extrapolation out of samples.

In this work we instead consider kernel-based \emph{meshfree} models constructed via  polyharmonic splines whose definition includes in particular the cubic RBF and the Thin Plate Spline (TPS); for further details see \cite{Fasshauer}. For the extrapolation issue, we take advantage of the use of the so-called Variably Scaled Kernels (VSKs) \cite{Bozzini1}, which might lead to more stable  and accurate schemes \cite{CAMPAGNA202030}. Here the VSKs are introduced to define a feature augmentation strategy (see e.g. \cite{li,Scholkopf}) with the aim to   realize \emph{accurate} extrapolation in larger neighbouring subsets.
For strictly positive definite kernels, error bounds for VSK interpolants can be found in \cite{vskmpi,vskjump}. Here we extend the VSKs to work with strictly {\em conditionally} positive kernels and we provide error bounds for the TPS-VSK interpolant in Beppo-Levi spaces. Furthermore, since data might be affected by measurement errors and uncertainty, we look for a regularized solution by ridge regression (see e.g. \cite{Wabba}). In doing so we focus on samples that decay rationally or exponentially.

The need to define data-driven rational or exponential models arises in several contexts, such as %investment growth, 
radioactive decay, atmospheric pressure changes, epidemic growth patterns from infectious disease outbreak data. To model situations in which the decay begins rapidly and then slows down to get closer and closer to zero, generalized spline models have already been investigated  \cite{Campagna201986,RC2012},
so as the exponential regression \cite{CampagnaMasiNMR}. Here the extrapolation  is driven  by data through the definition  of a scaling function   which is computed via a preliminary non-linear fitting of the available data.
%Numerical experiments are carried out by considering
%parametric fitting functions that decay rationally or exponentially.
%In \cite{Engelsen2018} different quantitative strategies, namely ratio fitting, discrete exponential fitting,  least squares regression, and inverse Laplace transformation combined with regression are compared in the quantitative analysis of time domain NMR relaxation data. 
Our experiments show that the so-constructed method can be effectively used for the extrapolation problem and point out that the proposed tool might be seen as an  alternative to the sometimes computationally demanding Support Vector Regression (SVR), that might be considered as the \emph{state of the art} for regression purposes; refer e.g. to  \cite{Fasshauer15}.   

The paper is organized as follows. In Section \ref{preliminaries} we briefly review the basics of kernel-based interpolants. In Section \ref{main_section}, we investigate the VSKs based on polyharmonic splines and we provide error bounds. In Section \ref{FA} we further discuss our feature augmentation strategy and the related definition   of a scaling function. Section \ref{num} deals with numerical experiments, while conclusions and future work are outlined in Section \ref{concl}.

\section{Kernel framework}
\label{preliminaries}

Let $ {X} = \{  \boldsymbol{x}_i, \; i = 1,  \ldots , n\} \subseteq \Omega$, $\Omega \subseteq 
\mathbb{R}^{d}$, be a set of distinct scattered points and $ {F}= \{ f_i =
f(\boldsymbol{x}_i) , \; i=1, \ldots, n \} \subseteq \mathbb{R}$ be the associated data values, sampled from an unknown function $f:\Omega \longrightarrow \mathbb{R}$, we define an approximating model for such samples %$(\boldsymbol{x}_i,f_i)_{i=1,\ldots,n}$  
 and empirically study its effectiveness for the extrapolation  outside the domain $\Omega$. 
A kernel-based interpolant, defined on a set   $\Lambda \subseteq 
\mathbb{R}^{d}$, is a function $P_f: \Lambda \supseteq \Omega \longrightarrow \mathbb{R}$  %so that $P_f(\boldsymbol{x}_i) = f_i$, $i=1,\ldots,n$,  
 of the form \cite{Fasshauer}:
\begin{equation}\label{1bis}
P_f\left( \boldsymbol{x}\right)= \sum_{k=1}^{n} \alpha_k \kappa \left( \boldsymbol{x} , \boldsymbol{x}_k  \right) +  \sum_{j=1}^{m} \beta_{j} p_{j}\left(\boldsymbol{x}\right), \quad \boldsymbol{x} \in \Lambda,
\end{equation}
with 
  conditionally positive definite radial kernels $\kappa: \Lambda \times \Lambda \longrightarrow \mathbb{R}$ of order $l$ on $\mathbb{R}^d$  
 and $\{p_1, \ldots , p_m\}$ a basis for the $m$-dimensional  linear space $ \Pi_{l-1}^{d}$  of polynomials of total degree less than or  equal to $l-1$ in $d$ variables,
where
\begin{equation*}
m=
\begin{pmatrix} 
l -1+d \\ 
l-1 
\end{pmatrix}.
\end{equation*}
The coefficients $\boldsymbol{\alpha}= \left(\alpha_1, \ldots, \alpha_n\right)^{\intercal}$, and   $\boldsymbol{\beta}= (\beta_1, \ldots, \beta_m)^{\intercal}$,  are uniquely  identified  by imposing  the 
  $n$ interpolation conditions $P_f(\boldsymbol{x}_i) = f_i$, $i=1,\ldots,n$, and the $m$ constraints for polynomial reproduction \cite{Wendland05}. In other words, they are determined by solving 
\begin{equation}
\label{eq2}
\underbrace{
\begin{pmatrix}
\mA & \mP \\
\mP^{\intercal}  & \mO
\end{pmatrix}
 }_{\mK}
 \underbrace{
\begin{pmatrix}
\boldsymbol{\alpha}\\
\boldsymbol{\beta}
\end{pmatrix}
 }_{\boldsymbol{\gamma}}
=
 \underbrace{
\begin{pmatrix}
\boldsymbol{f}\\
\boldsymbol{0}
\end{pmatrix},
 }_{\boldsymbol{g}}
\end{equation}
where 
\begin{equation*}
\mA_{ik}= \kappa \left( \boldsymbol{x}_i , \boldsymbol{x}_k  \right), \quad i,k=1, \ldots, n,
\end{equation*}
\begin{equation*}
\mP_{ij}=p_{j}\left( \boldsymbol{x}_i \right), \quad i=1, \ldots, n, \quad j=1, \ldots, m.
\end{equation*}
Moreover,  $\boldsymbol{f} =\left(f_1, \ldots , f_n\right)^{\intercal}$,  $\boldsymbol{0}$ is a zero vector of length $m$ and $\mO$  is a zero matrix  $m \times m$. 
To assess the conditions under which such system admits a unique solution, we introduce the definition of \emph{unisolvent set}. 
\begin{definition}
		A set of points $X = \{  \boldsymbol{x}_i, i = 1,  \ldots , n \} \subseteq \Omega$ is called $\left(l-1\right)$-unisolvent if the only polynomial of total degree at most $l-1$ interpolating zero data on $X$ is the zero polynomial.
\end{definition}	
We remark that if $\kappa$ is a strictly conditionally positive definite function of order $l$ on $\mathbb{R}^{d}$ and the set $X$ of data points forms a $\left(l-1\right)$-\emph{unisolvent set}, the system \eqref{eq2} admits a unique solution; refer e.g. to \cite[Theorem 7.2]{Fasshauer}. 

Under the hypothesis of radial kernels,  there exists a function $ \varphi: [0, \infty) \to \mathbb{R}$, known as   Radial Basis Function (RBF), such that for all $\boldsymbol{x},\boldsymbol{y} \in \Lambda$ 
\begin{equation*}
\kappa(\boldsymbol{x},\boldsymbol{y})=\varphi( ||\boldsymbol{x}-\boldsymbol{y}||_2)=\varphi(r), \qquad r:=||\boldsymbol{x}-\boldsymbol{y}||_2.
\end{equation*}
Principally to overcome instability issues due to the ill-conditioning of the kernel matrix    (see e.g. \cite{Schaback1995a}), the so-called Variably Scaled Kernels (VSKs) \cite{Bozzini1} have been introduced. In particular, given a scaling function $\psi:\R^d\rightarrow \Sigma \subseteq \mathbb{R}$ we denote by   
\begin{equation}
    \label{Gpsi}
G_{\Psi}(\Lambda) = \{(\boldsymbol{{x}},\psi(\boldsymbol{x})) \ / \ \boldsymbol{{x}} \in \Lambda\} \subset \Lambda \times \Sigma, \end{equation}
the graph of the function $\psi$ with respect to $\Lambda$. 
Then, the VSK  $\kappa^{\Psi}: \Lambda \times \Lambda \longrightarrow \mathbb{R}$ is defined as a standard kernel $ \kappa: G_{\Psi}(\Lambda) \times G_{\Psi}(\Lambda) \longrightarrow \mathbb{R}$, i.e.
	\begin{equation}\label{def_vsk_eq}
	\kappa^{\Psi}(\boldsymbol{{x}},\boldsymbol{{y}}):= \kappa((\boldsymbol{{x}},\psi(\boldsymbol{x})),(\boldsymbol{y},\psi(\boldsymbol{y})),
	\end{equation}
	for $\boldsymbol{x},\boldsymbol{y}\in\Lambda$.
	
Suitable choices of the scaling function $\psi$ might improve stability and preserve shape properties of the original function, 
see  \cite{vskmpi,vskjump,romani,rossini}. In the next section,    we will focus on specific kernels, and will use the VSKs as feature augmentation strategy for the extrapolation issue. 

\section{Feature augmentation for polyharmonic splines}
\label{main_section}

The selected kernel bases are the so-called \emph{polyharmonic splines}, 
  defined as 
\begin{equation} \label{PS}
\kappa_{d,l}(\boldsymbol{x},\boldsymbol{y})= \varphi_{d,l}(r)=\left\{
    \begin{array}{cc}
        r^{2l-d}, & \quad \textrm{for $d$ odd},  \\
        r^{2l-d} \log{r}, &\quad  \textrm{for $d$ even},  \\
    \end{array} \right.
\end{equation}
for $\boldsymbol{x},\boldsymbol{y} \in \Lambda$ and $2l>d$, with $d$ the space dimension. They have been introduced by J. Duchon, R. Harder and  R.N. Desmarais the 1970s; refer to \cite{Duchon76,Harder}. Such functions are strictly conditionally  positive definite of order $l$. Moreover, the condition number of the approximation problem induced by polyharmonic splines is invariant under rotations, translations and uniform scalings; see \cite[Theorem 4.5]{Iske}. 

Concerning the VSKs seen in the context of feature augmentation tools, if we suppose to have some a priori information on the samples, e.g. the knowledge about    the asymptotic behaviour of the function or about the steep gradients, we can directly encode such   information into the kernel. Indeed, the scaling function $\psi$ might be selected so that it \emph{mimics} the samples.   In the following, we propose to use a non-linear fitting of $\psi: \mathbb{R}^d \longrightarrow \Sigma \subseteq \mathbb{R}$ acting as a feature augmentation rule
(see e.g. \cite{li,Taylor}).
So the VSK encodes the information through the new feature introduced by $\psi$.
 At this end,   we firstly   define a function $\Psi:\Lambda \longrightarrow G_{\Psi}(\Lambda)$ as
\begin{equation*} \label{ppsi}
\Psi(\boldsymbol{x}):=(\boldsymbol{x},\psi(\boldsymbol{x})), 
\end{equation*}
that  extends the data vector $\boldsymbol{x} \in \Omega$, including one more {\em feature}   depending on the original ones.
Then, equivalently to \eqref{def_vsk_eq}, in the VSK setting, we define the kernel $\kappa^{\Psi}: {\Lambda} \times \Lambda \longrightarrow \mathbb{R}$, given by
\begin{equation*}
\kappa^{\Psi}(\boldsymbol{x},\boldsymbol{y})= \kappa \left( \Psi(\boldsymbol{x}),\Psi(\boldsymbol{y}) \right),
\end{equation*}
where $\kappa: G_{\Psi}(\Lambda) \times G_{\Psi}(\Lambda) \longrightarrow \mathbb{R}$.

Being $G_{\Psi}(\Lambda) \subseteq \mathbb{R}^{d+1}$, we now   investigate the VSK interpolant for polyharmonic splines, whose definition depends on $d$ (see \cite{Bozzini1}).  
\begin{definition} \label{def0}
   Given $\Lambda \subseteq \mathbb{R}^{d}$, and let  $\psi: \mathbb{R}^d \longrightarrow \Sigma$  be the scaling function, the VSK interpolant $P_f^{\Psi}: \Lambda \longrightarrow \mathbb{R}$ is defined as $$P_f^{\Psi}(\boldsymbol{x})=P_f((\boldsymbol{x},\psi(\boldsymbol{x}))) \quad \boldsymbol{x} \in \Lambda,$$ where $(\boldsymbol{x},\psi(\boldsymbol{x})) \in G_{\Psi}(\Lambda)$ and  $P_f: G_{\Psi}(\Lambda) \longrightarrow \mathbb{R}$ and $G_{\Psi}$ are defined as in 
   \eqref{1bis} and (\ref{Gpsi}), respectively.
\end{definition}

Because of the equivalence in Definition \ref{def0} we will refer to both $P_f^{\Psi}: \Lambda \longrightarrow \mathbb{R}$ and $P_f: G_{\Psi}(\Lambda) \longrightarrow \mathbb{R}$ as VSK interpolants.  We better formalize this concept via the following proposition.

\begin{proposition} \label{prop1}
Let $\kappa_{d,l}$ be a polyharmonic spline with  $2l>d+1$. Let $\psi: \mathbb{R}^d \longrightarrow \Sigma \subseteq \mathbb{R}$ be the scaling function for the VSK setting and $G_{\Psi}(\Lambda)$ its associated graph on $\Lambda$. Given a set of scattered data $X=\{\boldsymbol{x}_1,\ldots,\boldsymbol{x}_n\} \subseteq \Omega$,  $\Omega \subseteq \mathbb{R}^d$, and the associated function values $F = \{f_1,\ldots,f_n\} \subseteq \mathbb{R}$, the coefficients $  \boldsymbol{\alpha}= \left(\alpha_1, \ldots, \alpha_n\right)^{\intercal}$ and $  \boldsymbol{\beta}= (\beta_1, \ldots, \beta_m)^{\intercal}$ for the VSK interpolant $P_f^{\Psi}: \Lambda \longrightarrow \mathbb{R}$  are given by the solution of a system of the form  \eqref{eq2}, with
\begin{equation*}
\mA_{ik}= \kappa_{d+1,l} \left( (\boldsymbol{x}_i, \psi(\boldsymbol{x}_i)), (\boldsymbol{x}_k, \psi(\boldsymbol{x}_k))\right), \quad i,k=1, \ldots, n,
\end{equation*}
\begin{equation*}
\mP_{ij}=p_{j}\left( (\boldsymbol{x}_i,\psi(\boldsymbol{x}_i)) \right), \quad i=1, \ldots, n, \quad j=1, \ldots, m, 
\end{equation*}
  $p_j \in  \Pi_{l-1}^{d+1}$, $j=1, \ldots, m,$ and
\begin{equation} \label{eqm}
m=
\begin{pmatrix} 
l +d \\ 
l-1 
\end{pmatrix}.
\end{equation}
\end{proposition}

\begin{proof}
From \eqref{def_vsk_eq} and \eqref{PS}, we obtain that for $\boldsymbol{x},\boldsymbol{y} \in \Lambda$ $$\kappa^{\Psi}_{d,l} \left(\boldsymbol{x},\boldsymbol{y}\right) =\kappa_{d+1,l} \left( (\boldsymbol{x},\psi(\boldsymbol{x})) , (\boldsymbol{y},\psi(\boldsymbol{y}))  \right).$$ 
Moreover, since for $\boldsymbol{x} \in \Lambda$, $P_f^{\Psi}(\boldsymbol{x})=P_f((\boldsymbol{x},\psi(\boldsymbol{x})))$, where  $(\boldsymbol{x},\psi(\boldsymbol{x})) \in G_{\Psi}(\Lambda)$, we have that:
\begin{equation}\label{5bis}
P_f^{\Psi}  \left( \boldsymbol{x}\right)= \sum_{k=1}^{n} \alpha_k \kappa_{d+1,l} \left( (\boldsymbol{x},\psi(\boldsymbol{x})) , (\boldsymbol{x}_k,\psi(\boldsymbol{x}_k))  \right) +  \sum_{j=1}^{m} \beta_{j} p_{j}\left((\boldsymbol{x},\psi(\boldsymbol{x}))\right), \quad 
\end{equation}
where $p_1, \ldots , p_m$, are a basis for $ \Pi_{l-1}^{d+1}$ with $m$ defined as in \eqref{eqm}. By imposing the interpolation conditions $P_f^{\Psi}(\boldsymbol{x}_i)=f_i$ the thesis follows. 
\end{proof}

To provide error bounds for the VSK interpolant constructed via polyharmonic splines, we first need to recall that to any kernel $\kappa$ we can associate the so-called native space, refer e.g. to \cite{Schaback1995a,Wendland05} for its definition. For our scopes, we only need to point out that the error analysis for polyharmonic splines is carried out in particular native spaces, precisely in the Beppo–Levi spaces, see \cite[p. 366]{deni}. They are also referred to as homogeneous Sobolev spaces. Letting $h \in \mathbb{Z}^+$, the Beppo-Levi space $\textrm{BL}_{h}(\mathbb{R}^d)$ is defined as the space of all tempered distributions $f$ on $\mathbb{R}^d$ such that $D^{\boldsymbol{\xi}}f \in L^2(\mathbb{R}^d)$ for all $\boldsymbol{\xi} \in \mathbb{R}^d$ so that $|\boldsymbol{\xi}|=h$. The associated seminorm is 
$$
|f|^2_{\textrm{BL}_{h}(\mathbb{R}^d)} = \sum_{|\boldsymbol{\xi}|=h} \dfrac{h!}{\xi_1! \cdots \xi_d!} ||D^{\boldsymbol{\xi}}f||_2^2. 
$$

Set  $l=2$ in \eqref{PS}, for $d=1$ and $2$ (which are the cases of interest for our numerical experiments), $\varphi_{d,l}$ is usually referred to as cubic kernel and Thin Plate Spline (TPS), respectively. Both kernels are strictly conditionally positive definite of order $2$,  which implies that we have to impose the polynomial  reproduction. 

Here we extend to VSKs  an existing result true for TPSs. The following  proposition    theoretically grants the effectiveness of the proposed model, giving  local upper bounds for the  VSK interpolant.

\begin{proposition}
Let $X=\{{x}_1,\ldots,{x}_n\} \subseteq \Omega$,  $\Omega \subseteq \mathbb{R}$, be a set of scattered data and $F = \{f_1,\ldots,f_n\} \subseteq \mathbb{R}$ the associated function values. Let $\kappa_{1,2}$ be the cubic kernel,  $\psi: \mathbb{R} \longrightarrow \Sigma \subseteq \mathbb{R}$   the scaling function for the VSK setting and $G_{\Psi}(\Lambda)$ its associated graph on $\Lambda$. Let $P_f: G_{\Psi}(\Lambda) \longrightarrow \mathbb{R}$ be the VSK interpolant.
There exists an absolute constant $C$ such that, on any closed triangle $ T$ corresponding to three nodes on $G_{\psi}(X)=\{({{x}_i},\psi({x}_i)) \ / \ {{x}_i} \in X\}$,  we obtain
$$ 
||f-P_f||_{\infty,T} \leq C \rho |f|_{\emph{\textrm{BL}}_{2}(\mathbb{R}^2)}, \quad f \in \emph{\textrm{BL}}_{2}(\mathbb{R}^2),
$$
where $\rho$ is the longest edge of $T$.
\end{proposition}
\begin{proof}
Given $X$ and the cubic kernel, for ${x} \in \Lambda$ the classical interpolant $P_f: \Lambda \longrightarrow \mathbb{R}$  assumes the form
\begin{equation*}
P_f(x) = \sum_{k=1}^{n} \alpha_k \kappa_{1,2} \left({x}, {x}_k \right) +  \sum_{j=1}^{2} \beta_{j} p_{j}\left({x}\right) =  \sum_{k=1}^n \alpha_k |{x} -{x}_k |^3 + \sum_{j=1}^{2} \beta_{j} p_{j}\left({x}\right),
\end{equation*}
where $p_1$ and $p_2$  are a basis for $ \Pi_{1}^{1}$. 
Following Proposition \ref{prop1}, from (\ref{5bis}), for ${x} \in \Lambda$ we have that the VSK interpolant $P^{\Psi}_f: \Lambda \longrightarrow \mathbb{R}$ is  
\begin{equation*}
\begin{split}
P_f^{\Psi}  \left( {x}\right) & = P_f(({x},\psi({x}))) \, = \sum_{k=1}^{n} \alpha_k \kappa_{2,2} \left( ({x},\psi({x})) , ({x}_k,\psi({x}))  \right) +  \sum_{j=1}^{3} \beta_{j} p_{j}\left(({x},\psi({x}))\right)= \\
 & =  \sum_{k=1}^n \alpha_k ||({x},\psi({x})) -({x}_i,\psi({x}_i)) ||_2^2 \log{(||({x},\psi({x})) -({x}_i,\psi({x}_i)) ||_2)} \,+\\
 & + \sum_{j=1}^3 \beta_j p_j(({x},\psi({x}))), 
\end{split}
\end{equation*}
where $p_1, p_2 , p_3$, are a basis for $ \Pi_{1}^{2}$. 
This is the standard setting for the TPS computed on scattered data and thanks to \cite[p. 176]{beatson} the thesis follows; refer also to \cite{powell}.
\end{proof}

Even if we provided error bounds for the interpolation setting, we have to point out that the interpolation conditions might be relaxed in some cases, for instance when the measurements are affected by noise and errors or when the main issue is to capture the trend of data. To accomplish this, we use a method
that is generally referred to as   ridge regression; see e.g. \cite{Fasshauer15,Wabba}. Precisely, to smooth out the noise for both standard and VSK interpolants, we introduce a 
penalty term weighted by a non-negative regularization parameter $\lambda \in \mathbb{R}^+$ and, starting from the system \eqref{eq2},  we compute the coefficients by solving:
\begin{equation}\label{regpb}
   (\mK+\lambda \mI) \boldsymbol{\gamma} =   \boldsymbol{g},
\end{equation}
where $\mI \in \mathbb{R}^{(n+m) \times (n+m)}$ is the identity matrix. Note that in the VSK setting the matrix $\mK$ is defined by $\mA$ and $\mP$ both set as in Proposition \ref{prop1}.

We conclude this section by observing that the procedure for constructing a VSK approximant requires the definition of the scaling function. In the following, we thus derive this setting from assumptions on data.
%, in order to better extract information and infer from the samples.

\section{Practical VSK setting} 
\label{FA}
 
Here we point out the empirical framework that we will use for the numerical tests. Precisely, we first discuss the selection of the scaling function and then we also report some basics of SVR, introduced for comparisons. \\

\subsection{The scaling function} 

As in the previous section we set $d=1$, and we assume that the model is defined on a set of samples that decay with an exponential or rational trend. The choice is justified by the interest towards   applications in which the acquired data show this trend, as already pointed out.
Given $X=\{x_1,\ldots,x_n\}\subseteq \Omega$, $\Omega \subseteq \mathbb{R}$, and $F=\{f_1,\ldots,f_n\}\subseteq \mathbb{R}$,  we consider two classes  of    rational and exponential parametric real functions:
\begin{eqnarray}
&\mathcal{R}=\left\{g:\Lambda \times \mathbb{R}^3\longrightarrow \mathbb{R}:\, g(x,\boldsymbol{\mu}) = \dfrac{x^{-\mu_1}}{x^{\mu_2} +\mu_3}\right\},   \label{ps2}\\
&\mathcal{E}=\left\{g:\Lambda \times \mathbb{R}^3\longrightarrow \mathbb{R}:\, g(x,\boldsymbol{\mu}) = \mu_1 x {\rm e}^{-\mu_2x}+\mu_3 {\rm e}^{-\mu_2x} \right\},   \label{ps1}
\end{eqnarray}
 where $\boldsymbol{\mu}=(\mu_1,\mu_3,\mu_3)^{\intercal}$. 
 
%and we assume $\psi\in \mathcal{R}\cup \mathcal{E}$.
The estimation of the parameters $\mu_1,\mu_2,\mu_3$ is carried out with an iterative reweighted least squares algorithm  for non-linear fit \cite{holland,seber} of samples $(x_i,f_i)_{i=1,\ldots,n}$. At each iteration, the robust weights downweight outliers, so that their influence becomes neglectable. Computationally speaking we use the \textsc{Matlab}\textsuperscript{\textregistered} function \texttt{nlinfit.m}. Once the \emph{optimal} parameters $\mu^*_i$, $i=1,\ldots,3$, are estimated from the samples for both classes $\mathcal{R}$ and $\mathcal{E}$ we define
$\psi_1({x})=g({x},\boldsymbol{\mu}^*)$ for $g \in \mathcal{R}$ and $\psi_2({x})=g({x},\boldsymbol{\mu}^*)$ for $g \in \mathcal{E}$.
Then, %Once the parameters are estimated for both $\psi_1$ and $\psi_2$, 
letting $\boldsymbol{\psi}_j=(\psi_j(x_1), \ldots, \psi_j(x_n))^{\intercal}$, $j=1,2$, we set $\psi \equiv \psi_s$, where
$$
s = \textrm{argmin}_{j=1,2}  ||\boldsymbol{f}-\boldsymbol{\psi}_j||_2.
$$
%with $\boldsymbol{\psi}_j=({\psi}_j(x_i))_{i=1}^n$ and $\boldsymbol{f}=(f_i)_{i=1}^n$, being $f_i=f(x_i)$.
Finally we define the VSK approximant  on the   set $\{(x_1,\psi(x_1)), \ldots, (x_n,\psi(x_n))\}$   by fixing $l=2$ in \eqref{PS}, i.e. via the TPS.

In the numerical experiments   we  compare our model with SVR that we briefly describe below for clarity and for making the paper self-contained.  

\subsection{Support Vector Regression}

Given $X=\{x_1,\ldots,x_n\}\subseteq \Omega$, $\Omega \subseteq \mathbb{R}$, and $F=\{f_1,\ldots,f_n\}\subseteq \mathbb{R}$, the SVR model $P_f: \Lambda \longrightarrow \mathbb{R}$ is constructed via a minimization problem. Precisely, it reduces to solving \cite{Fasshauer15}
\begin{equation*}
\min_{\boldsymbol{\alpha},{\boldsymbol{\alpha}^*} \in \mathbb{R}^n} \left[ \dfrac{1}{2} \sum_{i=1}^n \sum_{j=1}^n (\alpha_i- \alpha_i^*) (\alpha_j^*- \alpha_j) \kappa({x}_i, {x}_j)+ \epsilon \sum_{i=1}^n (\alpha_i^* + \alpha_i) -\sum_{i=1}^n f_i (\alpha_i^*- \alpha_i)\right], 
\end{equation*} 
subject to:
\begin{align}
& 0 \leq \alpha_i, \alpha_i^* \leq \zeta, \quad   i=1, \ldots, n, \nonumber\\
& \sum_{i=1}^n  (\alpha_i^*- \alpha_i)=0.  \nonumber
\end{align}
where $\zeta \geq 0$ represents the so-called \emph{trade-off parameter} and it is indeed a smoothing parameter. The \emph{hyper-parameter} $\epsilon \geq 0$ indicates the width of the \emph{tube} in which the samples can fall into without being counted as errors. 
%In other words it models the errors. 
From the Karush Kuhn Tucker conditions (see e.g. \cite{Nocedal}), we have \cite{Scholkopf}: 
\begin{equation*}
P_f(\boldsymbol{x})=  \sum_{i=1}^n  (\alpha_i^*- \alpha_i) \kappa({x}, {x}_i)+b,
\end{equation*} 
where
\begin{equation*}
b = \left\{ \begin{array}{ll}
&  f_i- \sum_{j=1}^n  (\alpha_j^*- \alpha_j) \kappa({x}_i, {x}_j) - \epsilon, \quad \textrm{for}   \hskip 0.1cm \alpha_i \in (0,\zeta),\nonumber\\
&  f_i- \sum_{j=1}^n  (\alpha_j^*- \alpha_j) \kappa({x}_i, {x}_j) + \epsilon, \quad \textrm{for}   \hskip 0.1cm \alpha^*_i \in (0,\zeta),\nonumber
\end{array} \right.
\end{equation*}
 is defined via an average over all candidates.

We know have all the ingredients for testing in the next section the proposed technique.
 
\section{Numerical experiments}
\label{num}

Based on   the assumed  data trend, we present results on   data sets generated by rational/exponential functions. Particularly, in accordance with previous studies  \cite{Campagna201986,CCC_LNCS}, we consider 
the following test functions taken from a database of Laplace transforms \cite{CAMPAGNA2020_AMC}:
 \begin{equation*}
 \begin{split}
     f_1(x) = \dfrac{1}{x(x+1)^2}, \hskip 1cm &
     f_2(x) = \dfrac{1}{x+1},\\
     f_3(x) = \dfrac{x}{(x^2+1)^2}, \hskip 1cm &
     f_4(x) = {\rm e}^{-2x}, \\
     f_5(x) = \arctan{\dfrac{20}{x}}, \hskip 1cm &
     f_6(x) = \dfrac{x}{x^2 + 1}. \\
 \end{split}
  \end{equation*}
 and we sample them on $\Omega=[a,b]$, set $a=0.1$ and $b=2$, with $30$ nodes which follow four different distributions, precisely quasi-uniform Halton points, Chebyshev nodes, random   and uniform points. For the extrapolation issue, we need to define the interval $\Lambda = [\Lambda_1,\Lambda_2]$. While $ \Lambda_1=a$, $\Lambda_2=b+0.1i$, $i=0,\ldots,10$. In this way we can numerically verify the robustness of our method as $\Lambda_2$ increases. Of course, dealing with {\em asymptotic} extrapolation, we expect that to larger $\Lambda_2$ correspond larger errors. Specifically, we evaluate the Root Mean Square Error (RMSE) on $s=40$ equispaced nodes $\bar{x}_i$, $i=1,\ldots,s$, on $\Lambda$, i.e. we compute 
 $$
 {\rm RMSE} = \left( \dfrac{\sum_{i=1}^{s} ( f(\bar{x}_i)-A(\bar{x}_i) )^2}{s}  \right)^{1/2},
 $$
 where $A$ is the approximant constructed via either $P_f: \Lambda \longrightarrow \mathbb{R}$ or $P_f^{\Psi}: \Lambda \longrightarrow \mathbb{R}$. 
 Moreover, the regularization parameter in (\ref{regpb}) is set as $\lambda=1{\rm e}-06$.
\begin{table}
    \centering
     \caption{The RMSE for different values of $\Lambda_2$ and different node distributions obtained via the cubic kernel and the TPS-VSKs for $f_1$.}
    \label{tab:1}
\begin{tabular}{cccccc}
\hline
\noalign{\smallskip}
& $\Lambda_2$ & Halton & Chebyshev & Random & Uniform \\
\noalign{\smallskip}
\hline
\noalign{\smallskip}
\multirow{10}{*}{Cubic} 
&2.00& 1.64{\rm e}-02 & 3.94{\rm e}-03 & 1.16{\rm e}-02 & 2.69{\rm e}-02 \\ 
&2.10& 1.24{\rm e}-02 & 3.93{\rm e}-03 & 8.71{\rm e}-03 & 2.30{\rm e}-02 \\ 
&2.20& 8.31{\rm e}-03 & 3.94{\rm e}-03 & 6.11{\rm e}-03 & 1.88{\rm e}-02 \\ 
&2.30& 4.53{\rm e}-03 & 4.00{\rm e}-03 & 3.95{\rm e}-03 & 1.46{\rm e}-02 \\ 
&2.40& 2.00{\rm e}-03 & 4.19{\rm e}-03 & 2.60{\rm e}-03 & 1.04{\rm e}-02 \\ 
&2.50& 3.45{\rm e}-03 & 4.57{\rm e}-03 & 2.69{\rm e}-03 & 6.62{\rm e}-03 \\ 
&2.60& 5.97{\rm e}-03 & 5.20{\rm e}-03 & 3.75{\rm e}-03 & 4.31{\rm e}-03 \\ 
&2.70& 8.34{\rm e}-03 & 6.09{\rm e}-03 & 5.11{\rm e}-03 & 5.08{\rm e}-03 \\ 
&2.80& 1.05{\rm e}-02 & 7.23{\rm e}-03 & 6.61{\rm e}-03 & 7.53{\rm e}-03 \\ 
&2.90& 1.25{\rm e}-02 & 8.58{\rm e}-03 & 8.23{\rm e}-03 & 1.02{\rm e}-02 \\ 
&3.00& 1.44{\rm e}-02 & 1.01{\rm e}-02 & 9.97{\rm e}-03 & 1.28{\rm e}-02 \\ 
\noalign{\smallskip}
\hline
\noalign{\smallskip}
\multirow{10}{*}{TPS-VSK} 
&2.00 & 1.70{\rm e}-03 & 1.40{\rm e}-03 & 1.39{\rm e}-03 & 2.88{\rm e}-03 \\ 
&2.10 & 1.21{\rm e}-03 & 1.40{\rm e}-03 & 1.02{\rm e}-03 & 2.32{\rm e}-03 \\ 
&2.20 & 1.02{\rm e}-03 & 1.48{\rm e}-03 & 7.41{\rm e}-04 & 1.91{\rm e}-03 \\ 
&2.30 & 1.45{\rm e}-03 & 1.70{\rm e}-03 & 6.25{\rm e}-04 & 1.95{\rm e}-03 \\ 
&2.40 & 2.31{\rm e}-03 & 2.13{\rm e}-03 & 7.06{\rm e}-04 & 2.57{\rm e}-03 \\ 
&2.50 & 3.39{\rm e}-03 & 2.77{\rm e}-03 & 8.94{\rm e}-04 & 3.61{\rm e}-03 \\ 
&2.60 & 4.65{\rm e}-03 & 3.59{\rm e}-03 & 1.11{\rm e}-03 & 4.93{\rm e}-03 \\ 
&2.70 & 6.06{\rm e}-03 & 4.55{\rm e}-03 & 1.34{\rm e}-03 & 6.43{\rm e}-03 \\ 
&2.80 & 7.62{\rm e}-03 & 5.65{\rm e}-03 & 1.58{\rm e}-03 & 8.10{\rm e}-03 \\ 
&2.90 & 9.31{\rm e}-03 & 6.87{\rm e}-03 & 1.83{\rm e}-03 & 9.91{\rm e}-03 \\ 
&3.00 & 1.11{\rm e}-02 & 8.20{\rm e}-03 & 2.07{\rm e}-03 & 1.19{\rm e}-02 \\ 
\noalign{\smallskip}
\hline
\end{tabular}
\end{table} 

\begin{table}
    \centering
    \caption{The RMSE for different values of $\Lambda_2$ and different node distributions obtained via the cubic kernel and the TPS-VSKs for $f_2$.}
    \label{tab:2}
\begin{tabular}{cccccc}
\hline
\noalign{\smallskip}
& $\Lambda_2$ & Halton & Chebyshev & Random & Uniform \\
\noalign{\smallskip}
\hline
\noalign{\smallskip}
\multirow{10}{*}{Cubic} 
&2.00 & 1.85{\rm e}-05 & 3.56{\rm e}-06 & 1.31{\rm e}-05 & 2.89{\rm e}-05 \\ 
&2.10 & 1.01{\rm e}-04 & 6.67{\rm e}-05 & 8.52{\rm e}-05 & 8.63{\rm e}-05 \\ 
&2.20 & 3.55{\rm e}-04 & 2.74{\rm e}-04 & 3.18{\rm e}-04 & 3.13{\rm e}-04 \\ 
&2.30 & 7.87{\rm e}-04 & 6.50{\rm e}-04 & 7.24{\rm e}-04 & 7.16{\rm e}-04 \\ 
&2.40 & 1.41{\rm e}-03 & 1.21{\rm e}-03 & 1.31{\rm e}-03 & 1.30{\rm e}-03 \\ 
&2.50 & 2.22{\rm e}-03 & 1.95{\rm e}-03 & 2.09{\rm e}-03 & 2.08{\rm e}-03 \\ 
&2.60 & 3.22{\rm e}-03 & 2.87{\rm e}-03 & 3.06{\rm e}-03 & 3.04{\rm e}-03 \\ 
&2.70 & 4.40{\rm e}-03 & 3.98{\rm e}-03 & 4.20{\rm e}-03 & 4.18{\rm e}-03 \\ 
&2.80 & 5.76{\rm e}-03 & 5.26{\rm e}-03 & 5.53{\rm e}-03 & 5.50{\rm e}-03 \\ 
&2.90 & 7.30{\rm e}-03 & 6.71{\rm e}-03 & 7.03{\rm e}-03 & 7.00{\rm e}-03 \\ 
&3.00 & 9.01{\rm e}-03 & 8.34{\rm e}-03 & 8.70{\rm e}-03 & 8.66{\rm e}-03 \\ 
\noalign{\smallskip}
\hline
\noalign{\smallskip}
\multirow{10}{*}{TPS-VSK}
&2.00 & 3.14{\rm e}-16 & 2.10{\rm e}-16 & 3.03{\rm e}-16 & 3.34{\rm e}-16 \\ 
&2.10 & 3.08{\rm e}-16 & 4.26{\rm e}-16 & 2.83{\rm e}-16 & 3.17{\rm e}-16 \\ 
&2.20 & 2.91{\rm e}-16 & 1.25{\rm e}-15 & 2.75{\rm e}-16 & 3.08{\rm e}-16 \\ 
&2.30 & 2.95{\rm e}-16 & 2.49{\rm e}-15 & 3.00{\rm e}-16 & 3.05{\rm e}-16 \\ 
&2.40 & 2.88{\rm e}-16 & 4.12{\rm e}-15 & 2.70{\rm e}-16 & 3.16{\rm e}-16 \\ 
&2.50 & 2.81{\rm e}-16 & 6.12{\rm e}-15 & 2.79{\rm e}-16 & 3.00{\rm e}-16 \\ 
&2.60 & 2.76{\rm e}-16 & 8.43{\rm e}-15 & 2.55{\rm e}-16 & 3.05{\rm e}-16 \\ 
&2.70 & 2.71{\rm e}-16 & 1.11{\rm e}-14 & 2.59{\rm e}-16 & 3.18{\rm e}-16 \\ 
&2.80 & 2.82{\rm e}-16 & 1.39{\rm e}-14 & 2.59{\rm e}-16 & 3.27{\rm e}-16 \\ 
&2.90 & 2.66{\rm e}-16 & 1.71{\rm e}-14 & 2.87{\rm e}-16 & 3.37{\rm e}-16 \\ 
&3.00 & 2.89{\rm e}-16 & 2.05{\rm e}-14 & 2.69{\rm e}-16 & 3.32{\rm e}-16 \\ 
\noalign{\smallskip}
\hline
\end{tabular}
\end{table} 

\begin{table}
    \centering
     \caption{The RMSE for different values of $\Lambda_2$ and different node distributions obtained via the cubic kernel and the TPS-VSKs for $f_3$.}
    \label{tab:3}
\begin{tabular}{cccccc}
\hline
\noalign{\smallskip}
& $\Lambda_2$ & Halton & Chebyshev & Random & Uniform \\
\noalign{\smallskip}
\hline
\noalign{\smallskip}
\multirow{10}{*}{Cubic} 
&2.00  & 1.76{\rm e}-05 & 3.05{\rm e}-06 & 2.16{\rm e}-05 & 2.24{\rm e}-05 \\ 
&2.10  & 1.56{\rm e}-04 & 1.02{\rm e}-04 & 1.32{\rm e}-04 & 1.29{\rm e}-04 \\ 
&2.20  & 5.43{\rm e}-04 & 4.17{\rm e}-04 & 4.85{\rm e}-04 & 4.78{\rm e}-04 \\ 
&2.30  & 1.19{\rm e}-03 & 9.78{\rm e}-04 & 1.09{\rm e}-03 & 1.08{\rm e}-03 \\ 
&2.40  & 2.11{\rm e}-03 & 1.79{\rm e}-03 & 1.96{\rm e}-03 & 1.94{\rm e}-03 \\ 
&2.50  & 3.29{\rm e}-03 & 2.86{\rm e}-03 & 3.09{\rm e}-03 & 3.07{\rm e}-03 \\ 
&2.60  & 4.72{\rm e}-03 & 4.18{\rm e}-03 & 4.47{\rm e}-03 & 4.44{\rm e}-03 \\ 
&2.70  & 6.39{\rm e}-03 & 5.73{\rm e}-03 & 6.08{\rm e}-03 & 6.05{\rm e}-03 \\ 
&2.80  & 8.29{\rm e}-03 & 7.50{\rm e}-03 & 7.92{\rm e}-03 & 7.88{\rm e}-03 \\ 
&2.90  & 1.04{\rm e}-02 & 9.48{\rm e}-03 & 9.97{\rm e}-03 & 9.92{\rm e}-03 \\ 
&3.00  & 1.27{\rm e}-02 & 1.16{\rm e}-02 & 1.22{\rm e}-02 & 1.22{\rm e}-02 \\ 
\noalign{\smallskip}
\hline
\noalign{\smallskip}
\multirow{10}{*}{TPS-VSK}
&2.00  & 4.85{\rm e}-05 & 1.38{\rm e}-05 & 1.09{\rm e}-04 & 8.30{\rm e}-05 \\ 
&2.10  & 7.44{\rm e}-05 & 5.39{\rm e}-05 & 1.16{\rm e}-04 & 9.25{\rm e}-05 \\ 
&2.20  & 1.88{\rm e}-04 & 1.52{\rm e}-04 & 1.88{\rm e}-04 & 1.87{\rm e}-04 \\ 
&2.30  & 3.47{\rm e}-04 & 2.80{\rm e}-04 & 3.16{\rm e}-04 & 3.37{\rm e}-04 \\ 
&2.40  & 5.33{\rm e}-04 & 4.22{\rm e}-04 & 4.71{\rm e}-04 & 5.18{\rm e}-04 \\ 
&2.50  & 7.35{\rm e}-04 & 5.66{\rm e}-04 & 6.39{\rm e}-04 & 7.16{\rm e}-04 \\ 
&2.60  & 9.45{\rm e}-04 & 7.04{\rm e}-04 & 8.13{\rm e}-04 & 9.21{\rm e}-04 \\ 
&2.70  & 1.16{\rm e}-03 & 8.31{\rm e}-04 & 9.85{\rm e}-04 & 1.13{\rm e}-03 \\ 
&2.80  & 1.37{\rm e}-03 & 9.42{\rm e}-04 & 1.15{\rm e}-03 & 1.34{\rm e}-03 \\ 
&2.90  & 1.58{\rm e}-03 & 1.04{\rm e}-03 & 1.31{\rm e}-03 & 1.54{\rm e}-03 \\ 
&3.00  & 1.79{\rm e}-03 & 1.12{\rm e}-03 & 1.47{\rm e}-03 & 1.74{\rm e}-03 \\ 
\noalign{\smallskip}
\hline
\end{tabular}
\end{table}

\begin{table}
    \centering
     \caption{The RMSE for different values of $\Lambda_2$ and different node distributions obtained via the cubic kernel and the TPS-VSKs for $f_4$.}
    \label{tab:4}
\begin{tabular}{cccccc}
\hline
\noalign{\smallskip}
& $\Lambda_2$ & Halton & Chebyshev & Random & Uniform \\
\noalign{\smallskip}
\hline
\noalign{\smallskip}
\multirow{10}{*}{Cubic}
&2.00 & 3.75{\rm e}-05 & 7.79{\rm e}-06 & 2.90{\rm e}-05 & 6.31{\rm e}-05 \\ 
&2.10 & 1.01{\rm e}-04 & 6.45{\rm e}-05 & 8.49{\rm e}-05 & 9.64{\rm e}-05 \\ 
&2.20 & 3.37{\rm e}-04 & 2.57{\rm e}-04 & 3.00{\rm e}-04 & 2.98{\rm e}-04 \\ 
&2.30 & 7.29{\rm e}-04 & 5.93{\rm e}-04 & 6.66{\rm e}-04 & 6.59{\rm e}-04 \\ 
&2.40 & 1.27{\rm e}-03 & 1.07{\rm e}-03 & 1.18{\rm e}-03 & 1.17{\rm e}-03 \\ 
&2.50 & 1.96{\rm e}-03 & 1.69{\rm e}-03 & 1.84{\rm e}-03 & 1.82{\rm e}-03 \\ 
&2.60 & 2.78{\rm e}-03 & 2.44{\rm e}-03 & 2.62{\rm e}-03 & 2.60{\rm e}-03 \\ 
&2.70 & 3.72{\rm e}-03 & 3.30{\rm e}-03 & 3.53{\rm e}-03 & 3.51{\rm e}-03 \\ 
&2.80 & 4.78{\rm e}-03 & 4.28{\rm e}-03 & 4.55{\rm e}-03 & 4.52{\rm e}-03 \\ 
&2.90 & 5.93{\rm e}-03 & 5.35{\rm e}-03 & 5.67{\rm e}-03 & 5.63{\rm e}-03 \\ 
&3.00 & 7.18{\rm e}-03 & 6.51{\rm e}-03 & 6.87{\rm e}-03 & 6.84{\rm e}-03 \\  
\noalign{\smallskip}
\hline
\noalign{\smallskip}
\multirow{10}{*}{TPS-VSK} 
&2.00 & 5.64{\rm e}-16 & 2.03{\rm e}-16 & 8.27{\rm e}-16 & 6.69{\rm e}-16 \\ 
&2.10 & 4.32{\rm e}-15 & 2.68{\rm e}-15 & 4.86{\rm e}-15 & 3.67{\rm e}-15 \\ 
&2.20 & 1.28{\rm e}-14 & 8.38{\rm e}-15 & 1.46{\rm e}-14 & 1.11{\rm e}-14 \\ 
&2.30 & 2.50{\rm e}-14 & 1.67{\rm e}-14 & 2.87{\rm e}-14 & 2.19{\rm e}-14 \\ 
&2.40 & 4.02{\rm e}-14 & 2.73{\rm e}-14 & 4.66{\rm e}-14 & 3.55{\rm e}-14 \\ 
&2.50 & 5.82{\rm e}-14 & 3.97{\rm e}-14 & 6.76{\rm e}-14 & 5.15{\rm e}-14 \\ 
&2.60 & 7.85{\rm e}-14 & 5.38{\rm e}-14 & 9.14{\rm e}-14 & 6.97{\rm e}-14 \\ 
&2.70 & 1.01{\rm e}-13 & 6.93{\rm e}-14 & 1.18{\rm e}-13 & 8.96{\rm e}-14 \\ 
&2.80 & 1.25{\rm e}-13 & 8.60{\rm e}-14 & 1.46{\rm e}-13 & 1.11{\rm e}-13 \\ 
&2.90 & 1.50{\rm e}-13 & 1.04{\rm e}-13 & 1.75{\rm e}-13 & 1.34{\rm e}-13 \\ 
&3.00 & 1.76{\rm e}-13 & 1.22{\rm e}-13 & 2.07{\rm e}-13 & 1.58{\rm e}-13 \\ 
\noalign{\smallskip}
\hline
\end{tabular}
\end{table} 

Tests have been carried out on a Intel(R) Core(TM) i7 CPU 4712MQ 2.13 GHz processor. 
The RMSEs obtained by considering both the standard approximation via the cubic kernel and the TPS-VSKs are reported in Tables \ref{tab:1}--\ref{tab:4} for the test functions $f_i$, $i=1,\ldots,4$, respectively. The augmented feature for the VSK setting is dynamically selected by fitting $f_1$ and $f_2$  with a \emph{rational decay}, while $f_3$ and $f_4$ are modelled by an \emph{exponential decay}. The reader should note that closer is   the fitting of   $\psi$ to the
behaviour of the true function $f$, smaller is the RMSE. Indeed, the method almost reaches the machine precision when the test function $f$ belongs to one of the classes ${\cal R}$ or ${\cal E}$, defined in \eqref{ps2} and \eqref{ps1}, respectively. This holds true for $f_2$ and $f_4$. 
 
Moreover, we note that both methods behave similarly for each of the selected node distributions. This is more in general a peculiarity of kernel methods that are robust for different data distributions.  However, in our experiments, the TPS-VSK method  usually outperforms the standard cubic extrapolation. For a graphical feedback on the absolute errors refer to Figure \ref{fig:1}. 
\begin{figure}
    \centering
    \includegraphics[scale=0.8]{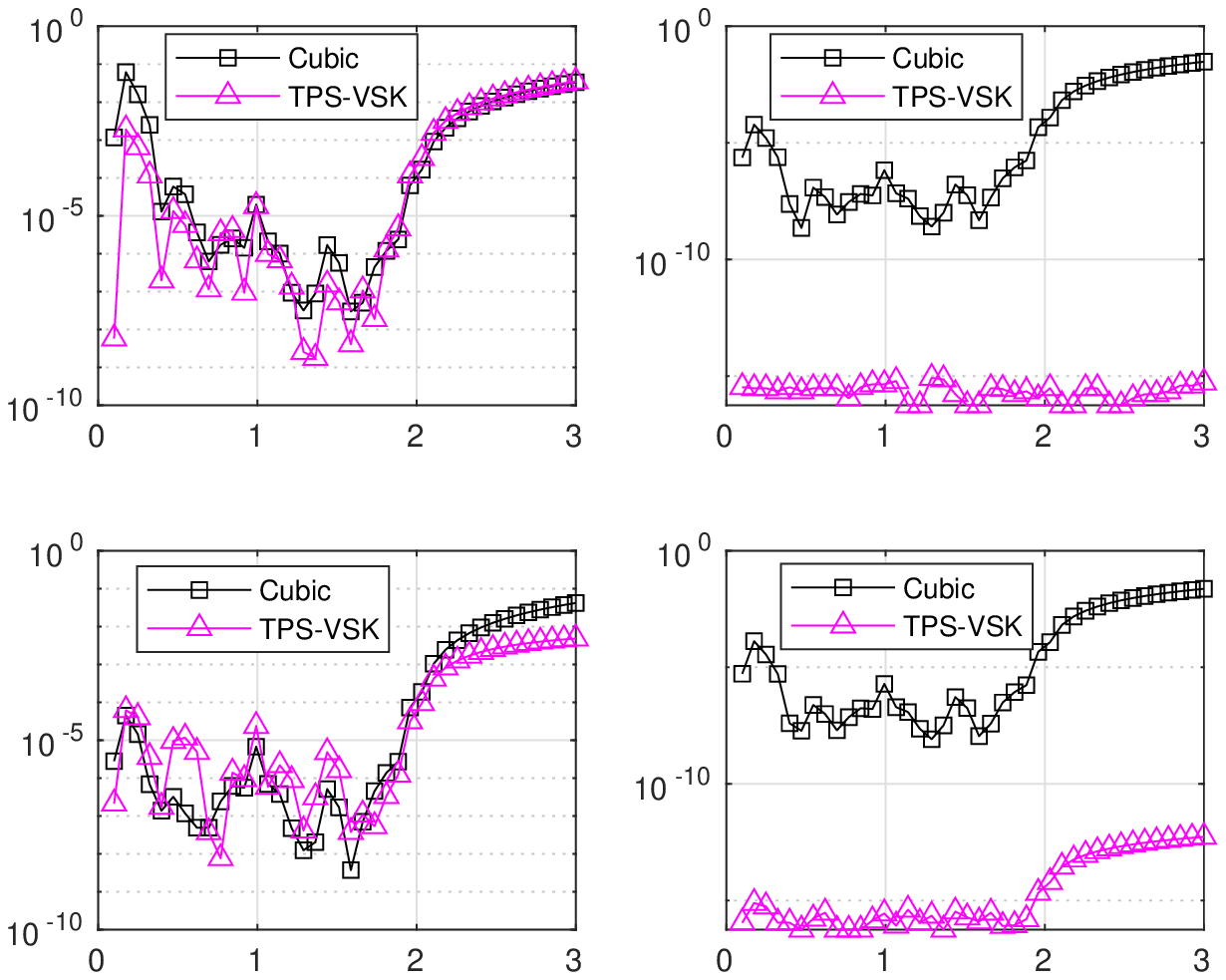}
    \caption{Left to right, top to bottom: the absolute errors in logarithmic scale, for  $\Lambda_2=3$ and Halton node distributions obtained via the cubic kernel ('$\square$') and the TPS-VSKs ('$\triangle$') for $f_1$, $f_2$, $f_3$ and $f_4$. }
    \label{fig:1}
\end{figure}

As last example, we take the functions $f_5$ and $f_6$ and we compare our varying scale setting with SVR. In this experiments, we also introduce noise on the measurements, i.e.  $f_i=f({x}_i )+ \delta_i$, $i=1,\ldots,n$. Precisely, we assume Gaussian white noise, i.e.  $\boldsymbol{\delta}=(\delta_1,\ldots,\delta_n )^{\intercal} \sim {\cal N}(0,\sigma^2 \mI)$, where $\mI$ is the $n \times n$ identity matrix and we fix $\sigma = 1{\rm e}-04$. 
We compare the VSK setting with the SVR trained with the cubic kernel and a standard $3$-fold validation for optimizing the hyperparameters (for this scope we use   the \textsc{Matlab}\textsuperscript{\textregistered} function {\tt fitrsvm.m}). 
Since we do not use any data-filling strategy for SVR, we only take into account equispaced data. The results are reported in Table \ref{tab:5}. For a graphical feedback, refer to Figure \ref{fig:5}. We note that, for the function $f_5$, which is approximately \emph{linear} on $\Lambda$, our results are comparable with SVR;  on the opposite, when \emph{learning} the function is not trivial, e.g. for $f_6$,  our model outperforms the standard SVR. However, we have to point out that the SVR performances could be improved by data assimilation procedures, i.e. one could construct a SVR model for each values of $\Lambda_2$. Nevertheless, this procedure would be too expensive if compared to the VSK strategy that only requires the computation of an additional feature. 

\begin{table}
    \centering
        \caption{The RMSE for different values of $\Lambda_2$ and uniform node distribution obtained via SVR and the TPS-VSKs for $f_5$ and $f_6$.}
    \label{tab:5}
\begin{tabular}{ccccc}
\hline
\noalign{\smallskip}
& \multicolumn{2}{c}{$f_5$} & \multicolumn{2}{c}{$f_6$} \\
\noalign{\smallskip}
$\Lambda_2$ & SVR & TPS-VSK & SVR & TPS-VSK \\
\noalign{\smallskip}
\hline
\noalign{\smallskip}
2.00 & 2.32{\rm e}-05 & 1.29{\rm e}-04 & 3.79{\rm e}-03 & 1.29{\rm e}-04 \\ 
2.10 & 2.35{\rm e}-05 &  1.32{\rm e}-04 & 6.09{\rm e}-03 & 1.31{\rm e}-04 \\ 
2.20 & 2.54{\rm e}-05  & 1.46{\rm e}-04  & 1.09{\rm e}-02  & 1.38{\rm e}-04 \\ 
2.30 & 2.93{\rm e}-05 & 1.73{\rm e}-04  & 1.84{\rm e}-02  & 1.49{\rm e}-04 \\ 
2.40 & 3.56{\rm e}-05 & 2.14{\rm e}-04 & 2.88{\rm e}-02  & 1.67{\rm e}-04 \\ 
2.50 & 4.42{\rm e}-05  & 2.64{\rm e}-04 & 4.24{\rm e}-02  & 1.83{\rm e}-04 \\ 
2.60 & 5.50{\rm e}-05  & 3.26{\rm e}-04  & 5.97{\rm e}-02 & 2.04{\rm e}-04 \\ 
2.70 & 6.78{\rm e}-05 & 3.98{\rm e}-04  & 8.11{\rm e}-02 & 2.26{\rm e}-04 \\ 
2.80 & 8.26{\rm e}-05 & 4.76{\rm e}-04 & 1.07{\rm e}-01  & 2.44{\rm e}-04 \\ 
2.90 & 9.94{\rm e}-05 & 5.68{\rm e}-04 & 1.38{\rm e}-01 & 2.68{\rm e}-04 \\ 
3.00 & 1.18{\rm e}-04  & 6.67{\rm e}-04 & 1.74{\rm e}-01 & 2.88{\rm e}-04 \\ 
\noalign{\smallskip}
\hline
\end{tabular}
\end{table}

\begin{figure}
    \centering
    \includegraphics[scale=0.8]{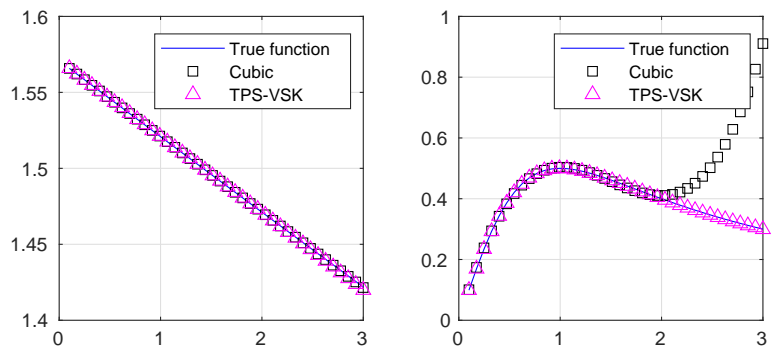}
    \caption{The graphical results for  $\Lambda_2=3$ and equispaced node distributions obtained via the cubic kernel  ('$\square$') and the TPS-VSKs ('$\triangle$') for $f_5$ and $f_6$, left and right respectively. }
    \label{fig:5}
\end{figure}

\section{Conclusion and work in progress}
\label{concl}

We investigated a novel procedure for the extrapolation issue based on the use of VSKs that serve as feature augmentation strategy. After extending them to strictly conditionally positive definite kernels and providing error bounds for the TPS in Beppo-Levi spaces, we tested the tool on several models. The results and in particular the comparison with SVR stress the benefits coming from the use of VSKs for the extrapolation issue.

Future work consists in extending the VSK setting to the context of SVR. The use of a scaling function indeed would introduce \emph{new feature maps} and \emph{spaces} which need further investigations.

 % (see Sect.~\ref{sec:1}).
%\paragraph{Paragraph headings} Use paragraph headings as needed.
 
%\begin{figure}
%\includegraphics{example.eps}
%\end{figure}
%\begin{figure*}
%  \includegraphics[width=0.75\textwidth]{example.eps}
%\end{figure*}
%\begin{table}
%\begin{tabular}{lll}
%\hline\noalign{\smallskip}
%\end{tabular}
%\end{table}

%\begin{acknowledgements}
%\end{acknowledgements}



% \bibliographystyle{unsrt}      
%
% \bibliography{bibEmma.bib}   
 % name your BibTeX data base

\end{document}